\documentclass[12pt]{amsart}
\usepackage{amsfonts,amssymb,amsmath,euscript}
    \newtheorem{thm}{Theorem}                     [section]
    \newtheorem{thm*}{Theorem}

    \newtheorem{cor}[thm]{Corollary}

    \newtheorem{lemma*}{Lemma}    %for AMS \newtheorem*
    \newtheorem{rems*}{Remark}   %for AMS \newtheorem*
\newcommand{\ndef}{\newcommand*}
\def\rndef{\renewcommand}

\ndef{\myaddress}[1]{\begin{center} \it\small #1 \end{center}}

\ndef{\clA}{{\mathcal A}} \ndef{\rmA}{{\mathrm A}} \ndef{\mbA}{{\mathbb A}} \ndef{\bfA}{{\mathbf A}} \ndef{\euA}{{\EuScript A}} \ndef{\frA}{{\mathfrak A}}
\ndef{\clB}{{\mathcal B}} \ndef{\rmB}{{\mathrm B}} \ndef{\mbB}{{\mathbb B}} \ndef{\bfB}{{\mathbf B}} \ndef{\euB}{{\EuScript B}} \ndef{\frB}{{\mathfrak B}}
\ndef{\clC}{{\mathcal C}} \ndef{\rmC}{{\mathrm C}} \ndef{\mbC}{{\mathbb C}} \ndef{\bfC}{{\mathbf C}} \ndef{\euC}{{\EuScript C}} \ndef{\frC}{{\mathfrak C}}
\ndef{\clD}{{\mathcal D}} \ndef{\rmD}{{\mathrm D}} \ndef{\mbD}{{\mathbb D}} \ndef{\bfD}{{\mathbf D}} \ndef{\euD}{{\EuScript D}} \ndef{\frD}{{\mathfrak D}}
\ndef{\clE}{{\mathcal E}} \ndef{\rmE}{{\mathrm E}} \ndef{\mbE}{{\mathbb E}} \ndef{\bfE}{{\mathbf E}} \ndef{\euE}{{\EuScript E}} \ndef{\frE}{{\mathfrak E}}
\ndef{\clF}{{\mathcal F}} \ndef{\rmF}{{\mathrm F}} \ndef{\mbF}{{\mathbb F}} \ndef{\bfF}{{\mathbf F}} \ndef{\euF}{{\EuScript F}} \ndef{\frF}{{\mathfrak F}}
\ndef{\clG}{{\mathcal G}} \ndef{\rmG}{{\mathrm G}} \ndef{\mbG}{{\mathbb G}} \ndef{\bfG}{{\mathbf G}} \ndef{\euG}{{\EuScript G}} \ndef{\frG}{{\mathfrak G}}
\ndef{\clH}{{\mathcal H}} \ndef{\rmH}{{\mathrm H}} \ndef{\mbH}{{\mathbb H}} \ndef{\bfH}{{\mathbf H}} \ndef{\euH}{{\EuScript H}} \ndef{\frH}{{\mathfrak H}}
\ndef{\clI}{{\mathcal I}} \ndef{\rmI}{{\mathrm I}} \ndef{\mbI}{{\mathbb I}} \ndef{\bfI}{{\mathbf I}} \ndef{\euI}{{\EuScript I}} \ndef{\frI}{{\mathfrak I}}
\ndef{\clJ}{{\mathcal J}} \ndef{\rmJ}{{\mathrm J}} \ndef{\mbJ}{{\mathbb J}} \ndef{\bfJ}{{\mathbf J}} \ndef{\euJ}{{\EuScript J}} \ndef{\frJ}{{\mathfrak J}}
\ndef{\clK}{{\mathcal K}} \ndef{\rmK}{{\mathrm K}} \ndef{\mbK}{{\mathbb K}} \ndef{\bfK}{{\mathbf K}} \ndef{\euK}{{\EuScript K}} \ndef{\frK}{{\mathfrak K}}
\ndef{\clL}{{\mathcal L}} \ndef{\rmL}{{\mathrm L}} \ndef{\mbL}{{\mathbb L}} \ndef{\bfL}{{\mathbf L}} \ndef{\euL}{{\EuScript L}} \ndef{\frL}{{\mathfrak L}}
\ndef{\clM}{{\mathcal M}} \ndef{\rmM}{{\mathrm M}} \ndef{\mbM}{{\mathbb M}} \ndef{\bfM}{{\mathbf M}} \ndef{\euM}{{\EuScript M}} \ndef{\frM}{{\mathfrak M}}
\ndef{\clN}{{\mathcal N}} \ndef{\rmN}{{\mathrm N}} \ndef{\mbN}{{\mathbb N}} \ndef{\bfN}{{\mathbf N}} \ndef{\euN}{{\EuScript N}} \ndef{\frN}{{\mathfrak N}}
\ndef{\clO}{{\mathcal O}} \ndef{\rmO}{{\mathrm O}} \ndef{\mbO}{{\mathbb O}} \ndef{\bfO}{{\mathbf O}} \ndef{\euO}{{\EuScript O}} \ndef{\frO}{{\mathfrak O}}
\ndef{\clP}{{\mathcal P}} \ndef{\rmP}{{\mathrm P}} \ndef{\mbP}{{\mathbb P}} \ndef{\bfP}{{\mathbf P}} \ndef{\euP}{{\EuScript P}} \ndef{\frP}{{\mathfrak P}}
\ndef{\clQ}{{\mathcal Q}} \ndef{\rmQ}{{\mathrm Q}} \ndef{\mbQ}{{\mathbb Q}} \ndef{\bfQ}{{\mathbf Q}} \ndef{\euQ}{{\EuScript Q}} \ndef{\frQ}{{\mathfrak Q}}
\ndef{\clR}{{\mathcal R}} \ndef{\rmR}{{\mathrm R}} \ndef{\mbR}{{\mathbb R}} \ndef{\bfR}{{\mathbf R}} \ndef{\euR}{{\EuScript R}} \ndef{\frR}{{\mathfrak R}}
\ndef{\clS}{{\mathcal S}} \ndef{\rmS}{{\mathrm S}} \ndef{\mbS}{{\mathbb S}} \ndef{\bfS}{{\mathbf S}} \ndef{\euS}{{\EuScript S}} \ndef{\frS}{{\mathfrak S}}
\ndef{\clT}{{\mathcal T}} \ndef{\rmT}{{\mathrm T}} \ndef{\mbT}{{\mathbb T}} \ndef{\bfT}{{\mathbf T}} \ndef{\euT}{{\EuScript T}} \ndef{\frT}{{\mathfrak T}}
\ndef{\clU}{{\mathcal U}} \ndef{\rmU}{{\mathrm U}} \ndef{\mbU}{{\mathbb U}} \ndef{\bfU}{{\mathbf U}} \ndef{\euU}{{\EuScript U}} \ndef{\frU}{{\mathfrak U}}
\ndef{\clV}{{\mathcal V}} \ndef{\rmV}{{\mathrm V}} \ndef{\mbV}{{\mathbb V}} \ndef{\bfV}{{\mathbf V}} \ndef{\euV}{{\EuScript V}} \ndef{\frV}{{\mathfrak V}}
\ndef{\clW}{{\mathcal W}} \ndef{\rmW}{{\mathrm W}} \ndef{\mbW}{{\mathbb W}} \ndef{\bfW}{{\mathbf W}} \ndef{\euW}{{\EuScript W}} \ndef{\frW}{{\mathfrak W}}
\ndef{\clX}{{\mathcal X}} \ndef{\rmX}{{\mathrm X}} \ndef{\mbX}{{\mathbb X}} \ndef{\bfX}{{\mathbf X}} \ndef{\euX}{{\EuScript X}} \ndef{\frX}{{\mathfrak X}}
\ndef{\clY}{{\mathcal Y}} \ndef{\rmY}{{\mathrm Y}} \ndef{\mbY}{{\mathbb Y}} \ndef{\bfY}{{\mathbf Y}} \ndef{\euY}{{\EuScript Y}} \ndef{\frY}{{\mathfrak Y}}
\ndef{\clZ}{{\mathcal Z}} \ndef{\rmZ}{{\mathrm Z}} \ndef{\mbZ}{{\mathbb Z}} \ndef{\bfZ}{{\mathbf Z}} \ndef{\euZ}{{\EuScript Z}} \ndef{\frZ}{{\mathfrak Z}}

\ndef{\tA}{{\widetilde A}} \ndef{\tcA}{{\widetilde\clA}} \ndef{\ttcA}{\widetilde{\tcA}} \ndef{\sfA}{{\textsf A}} \ndef{\ttA}{\widetilde{\tA}} \ndef{\dzA}{{A^\sharp}}
\ndef{\tB}{{\widetilde B}} \ndef{\tcB}{{\widetilde\clB}} \ndef{\ttcB}{\widetilde{\tcB}} \ndef{\sfB}{{\textsf B}} \ndef{\ttB}{\widetilde{\tB}} \ndef{\dzB}{{B^\sharp}}
\ndef{\tC}{{\widetilde C}} \ndef{\tcC}{{\widetilde\clC}} \ndef{\ttcC}{\widetilde{\tcC}} \ndef{\sfC}{{\textsf C}} \ndef{\ttC}{\widetilde{\tC}} \ndef{\dzC}{{C^\sharp}}
\ndef{\tD}{{\widetilde D}} \ndef{\tcD}{{\widetilde\clD}} \ndef{\ttcD}{\widetilde{\tcD}} \ndef{\sfD}{{\textsf D}} \ndef{\ttD}{\widetilde{\tD}} \ndef{\dzD}{{D^\sharp}}
\ndef{\tE}{{\widetilde E}} \ndef{\tcE}{{\widetilde\clE}} \ndef{\ttcE}{\widetilde{\tcE}} \ndef{\sfE}{{\textsf E}} \ndef{\ttE}{\widetilde{\tE}} \ndef{\dzE}{{E^\sharp}}
\ndef{\tF}{{\widetilde F}} \ndef{\tcF}{{\widetilde\clF}} \ndef{\ttcF}{\widetilde{\tcF}} \ndef{\sfF}{{\textsf F}} \ndef{\ttF}{\widetilde{\tF}} \ndef{\dzF}{{F^\sharp}}
\ndef{\tG}{{\widetilde G}} \ndef{\tcG}{{\widetilde\clG}} \ndef{\ttcG}{\widetilde{\tcG}} \ndef{\sfG}{{\textsf G}} \ndef{\ttG}{\widetilde{\tG}} \ndef{\dzG}{{G^\sharp}}
\ndef{\tH}{{\widetilde H}} \ndef{\tcH}{{\widetilde\clH}} \ndef{\ttcH}{\widetilde{\tcH}} \ndef{\sfH}{{\textsf H}} \ndef{\ttH}{\widetilde{\tH}} \ndef{\dzH}{{H^\sharp}}
\ndef{\tI}{{\widetilde I}} \ndef{\tcI}{{\widetilde\clI}} \ndef{\ttcI}{\widetilde{\tcI}} \ndef{\sfI}{{\textsf I}} \ndef{\ttI}{\widetilde{\tI}} \ndef{\dzI}{{I^\sharp}}
\ndef{\tJ}{{\widetilde J}} \ndef{\tcJ}{{\widetilde\clJ}} \ndef{\ttcJ}{\widetilde{\tcJ}} \ndef{\sfJ}{{\textsf J}} \ndef{\ttJ}{\widetilde{\tJ}} \ndef{\dzJ}{{J^\sharp}}
\ndef{\tK}{{\widetilde K}} \ndef{\tcK}{{\widetilde\clK}} \ndef{\ttcK}{\widetilde{\tcK}} \ndef{\sfK}{{\textsf K}} \ndef{\ttK}{\widetilde{\tK}} \ndef{\dzK}{{K^\sharp}}
\ndef{\tL}{{\widetilde L}} \ndef{\tcL}{{\widetilde\clL}} \ndef{\ttcL}{\widetilde{\tcL}} \ndef{\sfL}{{\textsf L}} \ndef{\ttL}{\widetilde{\tL}} \ndef{\dzL}{{L^\sharp}}
\ndef{\tM}{{\widetilde M}} \ndef{\tcM}{{\widetilde\clM}} \ndef{\ttcM}{\widetilde{\tcM}} \ndef{\sfM}{{\textsf M}} \ndef{\ttM}{\widetilde{\tM}} \ndef{\dzM}{{M^\sharp}}
\ndef{\tN}{{\widetilde N}} \ndef{\tcN}{{\widetilde\clN}} \ndef{\ttcN}{\widetilde{\tcN}} \ndef{\sfN}{{\textsf N}} \ndef{\ttN}{\widetilde{\tN}} \ndef{\dzN}{{N^\sharp}}
\ndef{\tO}{{\widetilde O}} \ndef{\tcO}{{\widetilde\clO}} \ndef{\ttcO}{\widetilde{\tcO}} \ndef{\sfO}{{\textsf O}} \ndef{\ttO}{\widetilde{\tO}} \ndef{\dzO}{{O^\sharp}}
\ndef{\tP}{{\widetilde P}} \ndef{\tcP}{{\widetilde\clP}} \ndef{\ttcP}{\widetilde{\tcP}} \ndef{\sfP}{{\textsf P}} \ndef{\ttP}{\widetilde{\tP}} \ndef{\dzP}{{P^\sharp}}
\ndef{\tQ}{{\widetilde Q}} \ndef{\tcQ}{{\widetilde\clQ}} \ndef{\ttcQ}{\widetilde{\tcQ}} \ndef{\sfQ}{{\textsf Q}} \ndef{\ttQ}{\widetilde{\tQ}} \ndef{\dzQ}{{Q^\sharp}}
\ndef{\tR}{{\widetilde R}} \ndef{\tcR}{{\widetilde\clR}} \ndef{\ttcR}{\widetilde{\tcR}} \ndef{\sfR}{{\textsf R}} \ndef{\ttR}{\widetilde{\tR}} \ndef{\dzR}{{R^\sharp}}
\ndef{\tS}{{\widetilde S}} \ndef{\tcS}{{\widetilde\clS}} \ndef{\ttcS}{\widetilde{\tcS}} \ndef{\sfS}{{\textsf S}} \ndef{\ttS}{\widetilde{\tS}} \ndef{\dzS}{{S^\sharp}}
\ndef{\tT}{{\widetilde T}} \ndef{\tcT}{{\widetilde\clT}} \ndef{\ttcT}{\widetilde{\tcT}} \ndef{\sfT}{{\textsf T}} \ndef{\ttT}{\widetilde{\tT}} \ndef{\dzT}{{T^\sharp}}
\ndef{\tU}{{\widetilde U}} \ndef{\tcU}{{\widetilde\clU}} \ndef{\ttcU}{\widetilde{\tcU}} \ndef{\sfU}{{\textsf U}} \ndef{\ttU}{\widetilde{\tU}} \ndef{\dzU}{{U^\sharp}}
\ndef{\tV}{{\widetilde V}} \ndef{\tcV}{{\widetilde\clV}} \ndef{\ttcV}{\widetilde{\tcV}} \ndef{\sfV}{{\textsf V}} \ndef{\ttV}{\widetilde{\tV}} \ndef{\dzV}{{V^\sharp}}
\ndef{\tW}{{\widetilde W}} \ndef{\tcW}{{\widetilde\clW}} \ndef{\ttcW}{\widetilde{\tcW}} \ndef{\sfW}{{\textsf W}} \ndef{\ttW}{\widetilde{\tW}} \ndef{\dzW}{{W^\sharp}}
\ndef{\tX}{{\widetilde X}} \ndef{\tcX}{{\widetilde\clX}} \ndef{\ttcX}{\widetilde{\tcX}} \ndef{\sfX}{{\textsf X}} \ndef{\ttX}{\widetilde{\tX}} \ndef{\dzX}{{X^\sharp}}
\ndef{\tY}{{\widetilde Y}} \ndef{\tcY}{{\widetilde\clY}} \ndef{\ttcY}{\widetilde{\tcY}} \ndef{\sfY}{{\textsf Y}} \ndef{\ttY}{\widetilde{\tY}} \ndef{\dzY}{{Y^\sharp}}
\ndef{\tZ}{{\widetilde Z}} \ndef{\tcZ}{{\widetilde\clZ}} \ndef{\ttcZ}{\widetilde{\tcZ}} \ndef{\sfZ}{{\textsf Z}} \ndef{\ttZ}{\widetilde{\tZ}} \ndef{\dzZ}{{Z^\sharp}}

\ndef{\bfa}{{\mathbf a}}
\ndef{\bfb}{{\mathbf b}}
\ndef{\bfc}{{\mathbf c}}
\ndef{\bfd}{{\mathbf d}}

\ndef{\euu}{{\EuScript u}}

  \ndef{\eps}{\varepsilon}
\let\leq\leqslant

\ndef{\lims}[1]{\lim\limits_{#1}}
\ndef{\sums}[1]{\sum\limits_{#1}}
\ndef{\ints}[1]{\int_{#1}}
\ndef{\sups}[1]{\sup\limits_{#1}}
\ndef{\liminfty}[1]{\lims{#1\to\infty}}
\ndef{\suminf}[1]{\sums{#1=1}^\infty}

\ndef{\limo}[1]{\omega\mbox{-}\!\!\!\lims{#1\to\infty}}          % \omega limit
\ndef{\limL}[1]{\rmL\mbox{-}\!\!\!\lims{#1\to\infty}}            % ``L" limit
\ndef{\limLOne}[1]{\clL_1\mbox{-}\!\lims{#1}}
\ndef{\tildelimo}[1]{\tilde\omega\mbox{-}\!\!\!\lims{#1\to\infty}}
\ndef{\slim}{\mathrm{s}\mbox{-}\!\!\lim}          % strong limit
\ndef{\wlim}{\mathrm{w}\mbox{-}\!\lim}          % strong limit

\ndef{\Aut}{\operatorname{Aut}}      % group of automorphisms
\ndef{\Ch}{\operatorname{ch}}        % Chern character
\ndef{\End}{\operatorname{End}}      % group of endomorphisms
\ndef{\Hom}{\operatorname{Hom}}      % group of homomorphisms
\rndef{\ker}{\operatorname{ker}}      % kernel of an operator
\ndef{\coker}{\operatorname{coker}}      % kernel of an operator
\ndef{\im}{\operatorname{im}}        % image of an operator
\ndef{\Log}{\operatorname{Log}}      % logarithm
\ndef{\OP}{\operatorname{OP}}        % abstract PDO's
\ndef{\Op}{\operatorname{Op}}        % abstract PDO's
\ndef{\Symb}{\operatorname{Symb}}    % symbol
\ndef{\Tr}{\operatorname{Tr}}        % usual trace
\ndef{\Wres}{\operatorname{Wres}}    % Wodzicki residue
\ndef{\cl}{\operatorname{cl}}        % Clifford
\ndef{\com}{\operatorname{com}}
\ndef{\const}{\operatorname{const}}  % constant
\ndef{\conv}{\operatorname{conv}}    % convex hull
\rndef{\det}{\operatorname{det}}     % determinant
\ndef{\detFK}[1]{\Delta\brs{#1}} % Fuglede-Kadison's determinant
\ndef{\detFKrel}[2]{\Delta_{#2}\brs{#1}} % Fuglede-Kadison's determinant

\ndef{\adj}{\operatorname{adj}}    % diagonal operator
\ndef{\diag}{\operatorname{diag}}    % diagonal operator
\ndef{\dist}{\operatorname{dist}}    % distance
\ndef{\dom}{\operatorname{dom}}      % domain
\ndef{\ec}{\operatorname{ec}}        % essential codimension
\ndef{\id}{1}                        % identity operator
\ndef{\ind}{\operatorname{ind}}      % index
\ndef{\mydeg}{\operatorname{deg}}    % degree of a diff. form
\ndef{\op}{\operatorname{op}}
\ndef{\rank}{\operatorname{rank}}
\ndef{\res}{\operatorname{res}}      % residue
\ndef{\rng}{\operatorname{ran}}      % range
\ndef{\sflow}{\operatorname{sf}}     % spectral flow
\ndef{\isf}{\operatorname{isf}}      % infinitesimal spectral flow
\ndef{\sign}{\operatorname{sign}}    % signum (a la C.-Ph.)
\ndef{\sgn}{\operatorname{sgn}}      % signum (a la Connes)
\ndef{\sing}{\operatorname{sing}}    % singular
\ndef{\supp}{\operatorname{supp}}    % support
\ndef{\tr}{\operatorname{tr}}        % trace
\ndef{\var}{\operatorname{var}}      % variation of measure
\ndef{\vol}{\operatorname{vol}}      % volume or volume form
\ndef{\wn}{\operatorname{wn}}        % winding number
\ndef{\wres}{\operatorname{wres}}    % Wodzicki residue
\rndef{\Im}{\operatorname{Im}}       % imaginary part of an operator
\rndef{\Re}{\operatorname{Re}}       % real part of an operator

\ndef{\prng}[1]{\mathrm R_{#1}} % {[\rng {#1}]}          % projection onto the range of #1
\ndef{\pker}[1]{\mathrm N_{#1}} % {[\ker {#1}]}          % projection onto the kernel of #1
\ndef{\rprng}[2]{\mathrm R_{#1}^{#2}}           % projection onto the relative range of #1
\ndef{\rpker}[2]{\mathrm N_{#1}^{#2}}           % projection onto the relative kernel of #1
\ndef{\rsupp}[1]{\supp_r(#1)}
\ndef{\lsupp}[1]{\supp_l(#1)}
\ndef{\rslv}[1]{R_z(#1)}      % resolvent
\ndef{\HH}{H}                 % initial operator of perturbation theory
\ndef{\tHH}{\tilde \HH}       % final operator of perturbation theory
\ndef{\VV}{V}                 % perturbation operator of perturbation theory
\ndef{\Rz}{R_z}               % resolvent of the initial operator
\ndef{\tRz}{\tR_z}            % resolvent of the final operator
\ndef{\psif}[1]{#1^{[1]}} % {\psi_{#1}}  % divided difference
\ndef{\WPlus}[1]{W_{#1}(\mbR)} %\ndef{\CPlus}[1]{C^{#1+}(\mbR)}
\newcommand{\xia}{\xi^{(a)}}
\newcommand{\xis}{\xi^{(s)}}

\ndef{\bndl}{\xi}                         % vector bundle
\ndef{\bndlA}{\eta}                       % vector bundle
\ndef{\GlueMap}{\varphi}                  % glue map of a bundle
\ndef{\ChartMap}{h}                       % chart diffeomorphism map of a manifold
\ndef{\chern}{\ensuremath{\mathrm{ch}}}
\ndef{\hilb}{\clH}                     % Hilbert space
\ndef{\hilba}{\clH^{(a)}}                    % absolutely continuous part of Hilbert space (wrt to a s.a. operator)
\ndef{\hilbs}{\clH^{(s)}}                    % singular part of Hilbert space (wrt to a s.a. operator)
   \ndef{\hilbasargument}{(\hilb)} %{(\hilb)}
\ndef{\LpH}[1]{\clL_{#1}\hilbasargument}       % the set of ...
\ndef{\saLpH}[1]{\clL_{sa}^{#1}\hilbasargument}       % the set of ...
\ndef{\clBH}{\clB\hilbasargument}              % the set of BOUNDED linear operators on Hilbert space
\ndef{\ubBH}{\clB_1\hilbasargument}            % the unit ball of the algebra of all BOUNDED linear operators on Hilbert space
\ndef{\clCH}{\clC\hilbasargument}              % the set of CLOSED DENSELY-DEFINED linear operators on Hilbert space
\ndef{\clKH}{\clK\hilbasargument}              % the set of COMPACT operators
\ndef{\clFH}{\clF\hilbasargument}              % the set of BOUNDED FREDHOLM operators
\ndef{\clUH}{\clU\hilbasargument}              % the set of UNITARIES on Hilbert space
\ndef{\clCFH}{{\clC\clF}\hilbasargument}       % the set of CLOSED DENSELY-DEFINED FREDHOLM OPERATORS on Hilbert space
\ndef{\saBH}{\clB_{sa}\hilbasargument}         % the set of S.-A. BOUNDED operators on Hilbert space
\ndef{\saCH}{\clC_{sa}\hilbasargument}         % the set of CLOSED DENSELY-DEFINED S.-A. operators on Hilbert space
\ndef{\saFH}{\clF_{sa}\hilbasargument}         % the set of BOUNDED FREDHOLM s.-a. operators
\ndef{\saKH}{\clK_{sa}\hilbasargument}         % the set of COMPACT S.-A. operators
\ndef{\saCFH}{\clC\clF_{sa}\hilbasargument}    % the set of CLOSED DENSELY-DEFINED S.-A. FREDHOLM operators on Hilbert space
\ndef{\clUFH}{\clU\clF\hilbasargument}         % the set of UNITARIES such that U+I is FREDHOLM
\ndef{\Uinj}{\clU_{inj}\hilbasargument}        % the set of UNITARIES such that U-I is injective
\ndef{\UFinj}{\clU\clF_{inj}\hilbasargument}   % the set of UNITARIES such that U-I is INJECTIVE and U+I is FREDHOLM

\ndef{\spproj}[2]{E^{#1}_{#2}}                      % spectral projection of #1
\ndef{\spprojb}[2]{E^{#2}_{#1}}                     % spectral projection of #1

\ndef{\LpN}[1]{\clL^{#1}(\clN,\tau)}     % noncommutative \mathcal L_p space
\ndef{\saLpN}[1]{\clL^{#1}_{sa}(\clN,\tau)} % s.-a. part of noncommutative \mathcal L_p space
\ndef{\rLpN}[1]{L^{#1}(\clN,\tau)}       % noncommutative L_p space
\ndef{\clAND}{(\clA,\clN,D)}             % spectral triple (A,N,D)
\ndef{\clBA}{{\clB(\clA)}}
\ndef{\saKN}{{\clK_{sa}(\clN,\tau)}}          % s.-a. \tau-compact operators
\ndef{\clKN}{{\clK(\clN,\tau)}}          % \tau-compact operators
\ndef{\clKtN}{{\clK(\tilde\clN,\tau)}}   % \tau-compact (maybe unbounded) operators
\ndef{\clFN}{{\clF(\clN,\tau)}}          % \tau-Fredholm operators
\ndef{\saFN}{{\clF_{sa}(\clN,\tau)}}     % self-adjoint \tau-Fredholm operators
\ndef{\clPN}{\clP(\clN)}                 % projections of N
\ndef{\clQN}{\clQ(\clN,\tau)}            % Calkin algebra N/K
\ndef{\infPN}{{\clP_\tau^\infty(\clN)}}  % infinite projections of N
\ndef{\clOF}[2]{\clF_{#1\mbox{-}#2}(\clN,\tau)}         % relatively Fredholm operators
\ndef{\oind}[2]{{\rm \tau\mbox{-}ind}_{#1\mbox{-}#2}}   % relative index
\ndef{\tind}{\tau\mbox{-}\ind}                  % semifinite index
\ndef{\DInd}{\ind_{\clD,\tau}}           % semifinite index
\ndef{\BF}{Breuer-Fredholm}              % Breuer-Fredholm
\ndef{\skewfred}[2]{$(#1\cdot #2)$ $\tau$\tire Fredholm}   % skew corner Fredholm
\ndef{\affl}{\eta}                       % affiliated
\ndef{\vNa}{von Neumann algebra}         % von Neumann algebra
\ndef{\nsf}{faithful normal semifinite } % normal semifinite faithful
\ndef{\taubrs}[1]{\tau\brackets{#1}}     % n.s.f. trace with brackets
\ndef{\sqbrs}[1]{[#1]}        % brackets
\ndef{\Sqbrs}[1]{\big[#1\big]}        % brackets
\ndef{\SqBrs}[1]{\Big[#1\Big]}        % brackets
\ndef{\domd}{\bigcap\limits_{n\ge 0} \dom\;\delta^n}         % domain of \delta^n's
\ndef{\DiffOP}{{\rm \clD}}
\ndef{\ADA}{\clA \cup [\clD,\clA]}
\ndef{\DixIdeal}[1]{\LpH{#1,\infty}}               % Dixmier ideal
\ndef{\dixideal}{\ell^{1,\infty}}                  % commutative Dixmier ideal
\ndef{\WDixIdeal}{\LpH{1,\mathrm w}}               % weak Dixmier ideal
\ndef{\DixIdealPos}[1]{\DixIdeal{#1}_+}            % positive part of Dixmier ideal
\ndef{\DixIdealN}[1]{\LpN{#1,\infty}}              % semifinite Dixmier ideal
\ndef{\DixIdealNPar}[2]{\clL^{#1,\infty}_{#2}(\clN,\tau)}    % semifinite Dixmier ideal
\ndef{\DixIdealNPos}[1]{\LpN{#1,\infty}_+}                   % positive part of semifinite Dixmier ideal
\ndef{\TrD}{\Tr_\omega}                                      % Dixmier trace
\ndef{\tauD}{{\tau_\omega}}                                  % semifinite Dixmier trace
\ndef{\ILogN}{\frac 1{\log(1+N)}}
\ndef{\DixNorm}[1]{\norm{#1}_{(1,\infty)}}                   % Dixmier norm
\ndef{\DixInt}[1]{\ints 0^t \mu_s(#1)\,ds}
\ndef{\DixIntL}[1]{\ints 0^{\lambda_{1/t}(#1)}\mu_s(#1)\,ds}
    \ndef{\SmallIdeal}{{\clL^{1, \mathrm w}}}
    \ndef{\SmallIdealMeas}{{\clL^{1, \mathrm w}_m}}
    \ndef{\DixIntII}[1]{\int_0^t \mu_s(#1)\,ds}
    \ndef{\DixIntf}[1]{\Phi_t(#1)}
    \ndef{\DixIntg}[1]{\Psi_t(#1)}

\ndef{\lpi}{\clL^{1,\pi}(\clN,\tau)}
\ndef{\strl}[1]{\stackrel \longrightarrow {#1}}
\ndef{\IIinfty}{$\mathrm{II}_\infty$\ }

\ndef{\fourier}[1]{\clF(#1)}          % Fourier transform of #1
\ndef{\HaarMeasBohrs}{\nu}            % Haar measure of the Bohr compact
\ndef{\BrownsMeas}{\mu}               % Brown's measure
\ndef{\BohrCont}[1]{\tilde{#1}}       % continuation of a function to the Bohr compact
\ndef{\APMean}{{M}}                   % mean value of a.p. function
\ndef{\CDSS}{{\clA_B}}                % Coburn-Douglas-Schaeffer-Singer's factor
\ndef{\matr}{{\rm Mat}}               % standard matrix algebra
\ndef{\seque}[1]{\ensuremath{\{#1_n\}_{n=1}^\infty}}    % sequence of numbers  a_1, a_2, ...
\ndef{\sequen}[2]{\ensuremath{\{#1_#2\}_{#2=1}^\infty}}    % sequence of numbers  a_1, a_2, ...
\ndef{\Seque}[1]{\ensuremath{\left(#1_0,#1_1,#1_2,\dots\right)}}    % sequence of numbers  a_1, a_2, ...
\ndef{\Cesaro}{H}                           % the Cesaro operator (on sequences)
\ndef{\CesaroRPlus}{M}                      % the Cesaro operator on positive semiaxis
\ndef{\Dilation}{D}                         % the dilation operator (on sequences)
\ndef{\Shift}{T}                            % the shift operator (on sequences)

\ndef{\norm}[1]{\left\Vert#1\right\Vert}    % norm of #1
\ndef{\TrNorm}[1]{\norm{#1}_1}              % trace norm of #1
\ndef{\HSNorm}[1]{\norm{#1}_2}              % Hilbert-Schmidt norm of #1
\ndef{\InftyNorm}[1]{\norm{#1}_\infty}      % uniform norm of #1
\ndef{\normQN}[1]{\norm{#1}_{\clQN}}        % Calkin norm of #1
\ndef{\clLpnorm}[2]{\norm{#2}_{\clL^{#1}}}    % $1,\infty$- trace norm of #1
\ndef{\clLnorm}[1]{\clLpnorm{1}{#1}}    % $1,\infty$- trace norm of #1

\ndef{\ccurve}{\gamma}                      % a curve in complex plane for Cauchy integral

\ndef{\abs}[1]{\left\lvert#1\right\rvert}   % absolute value of #1
\ndef{\set}[1]{\left\{#1\right\}}           % set of ...
\ndef{\brackets}[1]{\left(#1\right)}        % brackets
\ndef{\brs}[1]{\brackets{#1}}               % brackets
\ndef{\Brs}[1]{\big(#1\big)}                % brackets
\ndef{\BRS}[1]{\Big(#1\Big)}                % brackets
\ndef{\scal}[2]{\left\la #1,#2\right\ra}               % scalar product
\ndef{\precprec}{\prec\!\!\!\prec}
\ndef{\qeq}{\stackrel?=}
\ndef{\spectrum}[1]{\sigma_{#1}} %{\mathrm{Spec}(#1)}       % spectrum of an operator
\ndef{\spectruma}[1]{\sigma^{(a)}_{#1}} %{\mathrm{Spec}(#1)}       % absolutely continuous spectrum of an operator
\ndef{\numrange}[1]{\mathrm{W}(#1)}                         % numerical range of an operator
\rndef{\emptyset}{\varnothing}                              % empty set
\ndef{\csupp}{c}                           % subscript for compactly supported functions
\ndef{\closure}[1]{\overline{#1}}
\ndef{\linspan}[1]{\mathrm{span}\ {#1}}
\ndef{\bddborel}[1]{B(#1)}                 % the space of bounded Borel functions on the measure space #1
\ndef{\charfunc}{\chi}
\ndef{\FrDer}{\euD}                        % Fr\'echet derivative
\ndef{\LieDer}[1]{\pounds_{#1}\,}          % Lie derivative
\ndef{\dds}{\left.\frac d{ds} \right|_{s = 0}}
\ndef{\ortcmp}[1]{#1^{\scriptscriptstyle \perp}}            % orthogonal complement of projection #1
\ndef{\Laplace}{\Delta}                    % Laplace operator

\ndef{\matrPQ}[3]
{
    \left(
      \begin{array}{cc}
        #1_{11} & #1_{12} \\
        #1_{21} & #1_{22}
      \end{array}
    \right)_{[#2,#3]}
}
\ndef{\margOK}{\marginpar{\bf \small OK}}

\newcounter{margcomcount}
\ndef{\margcom}[1]{\marginpar{\bf \small #1} \addtocounter{margcomcount}{1}
   \index{\indexcom{{\bf COMMENT: #1}}}}

\newcounter{margproof}
\ndef{\margproof}{\marginpar{\bf \small PROOF} \addtocounter{margproof}{1}
  \index{**** \indexcom{{\bf PROOF}}}}

\newcounter{margdetails}
\ndef{\margdetails}{\marginpar{\bf Details} \addtocounter{margdetails}{1}
  \index{**** \indexcom{{\bf DETAILS}}}}

\newcounter{margproofb}
\ndef{\margproofb}[1]{\marginpar{\bf \small Proof(B) #1} \addtocounter{margproofb}{1}
  \index{**** \indexcom{{\bf PROOF(B): #1}}}}

\newcounter{margdetailsb}
\ndef{\margdetailsb}[1]{\marginpar{\bf \small Details(B)} \addtocounter{margdetailsb}{1}
  \index{**** \indexcom{{\bf DETAILS(B): \\ #1}}}}

\newcounter{margdetailsc}
\ndef{\margdetailsc}[1]{\marginpar{\bf \small Details(C)} \addtocounter{margdetailsc}{1}
  \index{**** \indexcom{{\bf DETAILS(C): \\ #1}}}}

\newcounter{margcomcountb}
\ndef{\margcomb}[1]{\marginpar{\bf \small #1} \addtocounter{margcomcountb}{1}
   \index{\indexcom{{\bf COMMENT(B): \\ #1}}}}

\ndef{\mytimes}{\!\times\!}
\ndef{\sss}[1]{\subsubsection{}\label{#1}}

\rndef{\phi}{\varphi}
\ndef{\OpenUnitDisk}{D}
\ndef{\RHS}{RHS}                            % right hand side
\ndef{\LHS}{LHS} %right and left hand side  % left hand side
\ndef{\ttt}{\Leftrightarrow}
\ndef{\then}{\Rightarrow}
\ndef{\tto}{\longrightarrow}
\ndef{\nno}{\nonumber\\}
\ndef{\newn}[1]{\index{#1} {\bfseries #1}}       % new notion
\ndef{\la}{\langle}
\ndef{\ra}{\rangle} \ndef{\dbar}{{\;\bar{\phantom{o}} \!\!\!\! d}}
\ndef{\stl}[1]{\stackrel{\vbox to 0pt{\vss\hbox{$\scriptstyle #1$}}}}
\ndef{\mathcomment}[1]{{\hfill \qquad\qquad\qquad\qquad\qquad\text{by (#1)}}}        % for comments at the ends of lines with math formulas
\ndef{\mathcomm}[1]{{\hfill \qquad\qquad\qquad\qquad\qquad\text{#1}}}        % for comments at the ends of lines with math formulas
\ndef{\details}[1]{\smallskip\begin{center} {\bf Here:}
#1\end{center}\medskip} \ndef{\indexcom}[1]{ --- #1}
\ndef{\longsim}{\ \sim \ }              % for use in formulas.
\ndef{\tire}{-}              % for use in formulas.
\ndef{\intinfinf}{\int_{-\infty}^\infty}
\ndef{\refnsftrace}{\cite[V.\,2.\,1]{TakI}} % reference to definition of n.f.s. trace %\cite[Definition V.2.1]{TakI}
\ndef{\refaffloper}{\cite[IV.\,5, Exercise 3]{TakI}} % reference to definition of affiliated operator
\ndef{\refsemifinvNa}{\cite[V.\,1.\,21]{TakI}} %  semifinite vNa
\ndef{\reftaumeasurable}{\cite[Definition 1.2]{FK86PJM}} % tau-measurable operator
\ndef{\reftautraceclassaffl}{\cite[V.2, p.\,320]{TakI}} % tau-trace class affiliated operator
\ndef{\refinvoperideal}{\cite[Appendix A.2]{CP2}} % invariant operator ideal
\ndef{\reftautracenorm}{\cite[V.2, p.\,320]{TakI}} % tau trace norm (1-norm)
\ndef{\reftaucompact}{\cite{}} % tau compact operator
\ndef{\reftauFredholm}{\cite[Appendix B]{PR94JFA}} % tau Fredholm operator

     \ndef{\npartial}{\slash\!\!\!\partial}
     \ndef{\Heis}{\operatorname{Heis}}
     \ndef{\Solv}{\operatorname{Solv}}
     \ndef{\Spin}{\operatorname{Spin}}
     \ndef{\SO}{\operatorname{SO}}
     \ndef{\Index}{\operatorname{index}}
             \ndef{\p}{\partial}
             \ndef{\dd}{|\clD|}
             \ndef{\n}{\parallel}

     \ndef{\gf}[2]{\genfrac{}{}{0pt}{}{#1}{#2}}
     \ndef{\ta}{\widetilde{\alpha}}
     \ndef{\tb}{\widetilde{\beta}}
     \ndef{\txi}{\widetilde{\xi}}
     \ndef{\tk}{\widetilde{K}}
     \ndef{\CGh}{\widetilde{\CG}}
     \ndef{\boe}{{\bf e}}\ndef{\bt}{{\bf t}}
     \ndef{\vth}{\vartheta}
     \ndef{\db}{\overline{\partial}}
     \ndef{\hV}{\hat{V}}
     \ndef{\cag}{{\clA^\Gamma}}
     \ndef{\sind}{\sigma{\rm -ind}}

\newcounter{slidecount}
\catcode`@=11
\newcommand{\slide}[2]{
  \newpage
  \addtocounter{slidecount}{1}
  \renewcommand{\@oddhead}{{\small Advanced Mathematics 1A -- 2009 \hfil Slide #1-\arabic{slidecount}}}
  \begin{center} \bf #2 \end{center}
}
\catcode`@=12

\let\LatexCite=\cite  % just renaming

\let\ifnumref\iffalse % use this command if you want LETTER REFERENCES like [CM]

\ndef{\ifuncited}[4]{\expandafter\ifx\csname used#4\endcsname\relax}

\ndef{\ifcited}[4]{\expandafter\ifx\csname used#4\endcsname\relax\else}
  \ndef{\papertitle}[1]{ \emph{#1}, }
  \ndef{\paperauthor}[2]{#2}  %{\ifnum#1=0$^*$\fi#2}
  \ndef{\pbbi}[9]{%
      \ifcited{#1}{#2}{#3}{#5}%
        \ifnumref%
          \bibitem{#5}\paperauthor{#1}{#6},\papertitle{#7}#8.%
        \else%
          \advance #9 by 1%
          \ifnum#9<1%
            \bibitem[#4]{#5}\paperauthor{#1}{#6}, \papertitle{#7}#8.%
          \else%
            \bibitem[#4$_{\the#9}$]{#5}\paperauthor{#1}{#6},\papertitle{#7}#8.%
          \fi%
        \fi%
      \fi%
  }
  \ndef{\mbbi}[8]{%
     \ifcited{#1}{#2}{#3}{#5}%
        \ifnumref%
          \bibitem{#5}\paperauthor{#1}{#6},\papertitle{#7}#8.%
        \else%
          \bibitem[#4]{#5}\paperauthor{#1}{#6},\papertitle{#7}#8.%
        \fi%
     \fi%
  }
\ndef{\AddCite}[1]{%
   \ifuncited{0}{0}{0}{#1}%
     \expandafter\gdef\csname used#1\endcsname {}%
   \fi%
}

\def\ProcessCite#1,{%
     \ifx\relax#1%
         \let\next=\relax%
     \else%
         \AddCite{#1}%
         \let\next=\ProcessCite%
     \fi%
     \next%
}

\ndef{\AddCites}[1]{\ProcessCite#1,\relax,}

\ndef{\CiteWithoutExtension}[1]{%
   \AddCites{#1}%
   \LatexCite{#1}%
}

\def\CiteWithExtension[#1]#2{%
   \AddCites{#2}%
   \LatexCite[#1]{#2}%
}

\ndef{\CleverCite}{%
    \ifx\NChar[ %
       \let\MyCite=\CiteWithExtension %
    \else %
       \let\MyCite=\CiteWithoutExtension %
    \fi %
    \MyCite%
}

\renewcommand{\cite}{\futurelet\NChar\CleverCite}
      \ndef{\volume}[1]{{\bf #1}}
      \ndef{\VolYearPP}[3]{\ifnum#2=0 (to appear)\else\volume{#1} (#2), #3\fi}
      \ndef{\VolNoYearPP}[4]{\ifnum#3=0 (to appear)\else\volume{#1} #2 (#3), #4\fi}
      \ndef{\libcode}[1]{}%{{,\bf\ #1}}

\ndef{\jnActaMath}[3]{Acta Math. \VolYearPP{#1}{#2}{#3}}                       % ActM        \libcode{510.5 A18}
\ndef{\jnAdvMath}[3]{Adv. in~Math. \VolYearPP{#1}{#2}{#3}}                     % AdvM        \libcode{510.5 A24}
\ndef{\jnAlgAnal}[3]{Algebra i~Analiz \VolYearPP{#1}{#2}{#3}}
\ndef{\jnAmerJMath}[3]{Amer. J. Math. \VolYearPP{#1}{#2}{#3}}                  % AJM         \libcode{}
\ndef{\jnAmerMathMonth}[3]{Amer. Math. Monthly \VolYearPP{#1}{#2}{#3}}         % AMM         \libcode{510.5 A3}
\ndef{\jnAnnMath}[4]{Ann. of~Math. \VolNoYearPP{#1}{#2}{#3}{#4}}               % AnnM        \libcode{510.5 A61}
\ndef{\jnAnalMath}[3]{J. Anal. Math. \VolYearPP{#1}{#2}{#3}}                   % JAnalM      \libcode{}
\ndef{\jnBullLondMathSoc}[3]{Bull. London Math. Soc. \VolYearPP{#1}{#2}{#3}}   % BLMS        \libcode{}
\ndef{\jnBullAMS}[3]{Bull. Amer. Math. Soc. \VolYearPP{#1}{#2}{#3}}   % BLMS        \libcode{}
\ndef{\jnCanMathBull}[3]{Canad. Math. Bull. \VolYearPP{#1}{#2}{#3}}            % CMB         \libcode{}
\ndef{\jnCanMath}[3]{Canad. J.~Math. \VolYearPP{#1}{#2}{#3}}             % CJM         \libcode{}
\ndef{\jnCommMathPhys}[3]{Comm. Math. Phys. \VolYearPP{#1}{#2}{#3}}             % CMP         \libcode{530.1505 C73}
\ndef{\jnCommPDE}[3]{Comm. Partial Differential Equations \VolYearPP{#1}{#2}{#3}}             % CMP         \libcode{530.1505 C73}
\ndef{\jnComptRendue}[3]{C.\,R.~Acad. Sci. Paris S\'er. A-B \VolYearPP{#1}{#2}{#3}}      % CR   \libcode{505 A 161}
\ndef{\jnContMath}[3]{Contemporary Math. \VolYearPP{#1}{#2}{#3}}               %
\ndef{\jnDukeMJ}[3]{Duke Math. J. \VolYearPP{#1}{#2}{#3}}
\ndef{\jnDiffGeom}[3]{J.~Diff. Geom. \VolYearPP{#1}{#2}{#3}}                   % JDG         \libcode{516.3605 J86}
\ndef{\jnErgodicTheory}[3]{Ergodic Theory and Dynamical Systems \VolYearPP{#1}{#2}{#3}} % ETDS  \libcode{}
\ndef{\jnFuncAnal}[3]{J.~Functional Analysis \VolYearPP{#1}{#2}{#3}}           % JFA         \libcode{515.05 J88}
\ndef{\jnFunkAnalPril}[4]{Funct. Anal. Appl. \VolNoYearPP{#1}{#2}{#3}{#4}}  % JFAP
\ndef{\jnGAFA}[3]{GAFA \VolYearPP{#1}{#2}{#3}}                                 % GAFA        \libcode{}
\ndef{\jnIHES}[3]{IHES Publ. Math. (Paris) \VolYearPP{#1}{#2}{#3}}             % IHES        \libcode{}
\ndef{\jnIEOT}[3]{Integral Equations Operator Theory   \VolYearPP{#1}{#2}{#3}} %         \libcode{515.4505 I61}
\ndef{\jnIsrMath}[3]{Israel J.~Math. \VolYearPP{#1}{#2}{#3}}                   % IMJ         \libcode{}
\ndef{\jnKTheory}[3]{K-Theory \VolYearPP{#1}{#2}{#3}}                          % KT          \libcode{}
\ndef{\jnLetMathPhys}[3]{Lett. Math. Phys. \VolYearPP{#1}{#2}{#3}}             % LMP         \libcode{}
\ndef{\jnMathAnn}[3]{Math. Ann. \VolYearPP{#1}{#2}{#3}}                        % MAnn        \libcode{510.5 M51}
\ndef{\jnMathAnalAppl}[3]{J.~Math. Anal. and Appl. \VolYearPP{#1}{#2}{#3}}     % MathAnalAppl  \libcode{510.5 J88} only in repository
\ndef{\jnMathNachr}[3]{Math. Nachr. \VolYearPP{#1}{#2}{#3}}
\ndef{\jnMathPhys}[3]{J. Math. Phys. \VolYearPP{#1}{#2}{#3}}
\ndef{\jnMathSocJap}[3]{J. Math. Soc. Japan \VolYearPP{#1}{#2}{#3}}
\ndef{\jnOperTheory}[3]{J.~Operator Theory \VolYearPP{#1}{#2}{#3}}             % JOT         \libcode{}
\ndef{\jnPacJMath}[3]{Pacific J.~Math. \VolYearPP{#1}{#2}{#3}}                  % PJM         \libcode{510.5 P11}
\ndef{\jnPositivity}[3]{Positivity \VolYearPP{#1}{#2}{#3}}
\ndef{\jnProcAmerMS}[3]{Proc. Amer. Math. Soc. \VolYearPP{#1}{#2}{#3}}         % PAMS        \libcode{510.5 A5p}
\ndef{\jnProcCambPhilSoc}[3]{Math. Proc. Camb. Phil. Soc. \VolYearPP{#1}{#2}{#3}}
\ndef{\jnReineAngew}[3]{J.~Reine Angew. Math. \VolYearPP{#1}{#2}{#3}}          % JRAM        \libcode{510.5 J86}
\ndef{\jnTokyoMath}[3]{Tokyo J.~Math. \VolYearPP{#1}{#2}{#3}}
\ndef{\jnTopology}[3]{Topology \VolYearPP{#1}{#2}{#3}}
\ndef{\jnTransAmerMathSoc}[3]{Trans. Amer. Math. Soc. \VolYearPP{#1}{#2}{#3}}
\ndef{\jnIzvANSSSR}[3]{Izv. Akad. Nauk SSSR, Ser. Mat. \VolYearPP{#1}{#2}{#3}}
\ndef{\jnIzvVyshUchZav}[3]{Izv. Vyssh. Uch. Zav., Mat. \VolYearPP{#1}{#2}{#3} (Russian)}
\ndef{\jnIzdatLenUniv}[2]{Izdat. Leningrad. Univ., Leningrad, (#1), #2 (Russian)}
\ndef{\jnFieldsInsComm}[3]{Fields Inst. Comm. \VolYearPP{#1}{#2}{#3}}
\ndef{\jnDoklANSSSR}[3]{Dokl. Akad. Nauk SSSR \VolYearPP{#1}{#2}{#3}}
\ndef{\jnMatZametki}[3]{Matem. zametki \VolYearPP{#1}{#2}{#3}}
\ndef{\jnRussMathSurvey}[3]{Russian Math. Surveys \VolYearPP{#1}{#2}{#3}}
\ndef{\jnSibMathJ}[3]{Sib. Math.~J. \VolYearPP{#1}{#2}{#3}}
\ndef{\jnSovMath}[3]{J.~Soviet math. \VolYearPP{#1}{#2}{#3}}
\ndef{\jnTransMoscMathSoc}[3]{Trans. Moscow Math. Soc. \VolYearPP{#1}{#2}{#3}}
\ndef{\jnUMN}[3]{Uspekhi Mat. Nauk \VolYearPP{#1}{#2}{#3}}

\ndef{\bkTransMathMon}[2]{Trans. Math. Monographs, AMS, \volume{#1}, #2}

\ndef{\pbBirkhauser}[1]{Birkh\"auser, Boston, #1}
\ndef{\pbFactorial}[1]{Moscow, Factorial, #1}
\ndef{\pbGauthier}[1]{Gauthier-Villars, Paris, #1}
\ndef{\pbNauka}[1]{Moscow, Nauka, #1 (Russian)}
\ndef{\pbNaukaR}[1]{Москва, Наука, #1}
\ndef{\pbPrinceton}[1]{Princeton University Press, Princeton, New Jersey, #1}
\ndef{\pbPublPerish}[1]{Publish or Perish Inc., Berkeley, #1}
\ndef{\pbSpringer}[1]{Springer-Verlag, #1}

\ndef{\myauthor}[1]{\mbox{#1}}

\ndef{\Agmon}{\myauthor{Sh.\,Agmon}}
\ndef{\Ahiezer}{\myauthor{N.\,I.\,Ahiezer}}
\ndef{\Arazy}{\myauthor{J.\,Arazy}}
\ndef{\Aronszajn}{\myauthor{N.\,Aronszajn}}
\ndef{\Astashkin}{\myauthor{S.\,V.\,Astashkin}}
\ndef{\Atiyah}{\myauthor{M.\,Atiyah}}
\ndef{\Avron}{\myauthor{J.\,E.\,Avron}}
\ndef{\Azamov}{\myauthor{N.\,A.\,Azamov}}
\ndef{\Banach}{\myauthor{S.\,Banach}}
\ndef{\Benameur}{\myauthor{M-T.\,Benameur}}
\ndef{\Bennett}{\myauthor{C.\,Bennett}}
\ndef{\Berezin}{\myauthor{F.\,A.\,Berezin}}
\ndef{\Berline}{\myauthor{N.\,Berline}}
\ndef{\Birman}{\myauthor{M.\,Sh.\,Birman}}
\ndef{\Blackadar}{\myauthor{B.\,Blackadar}}
\ndef{\Bogolyubov}{\myauthor{N.\,N.\,Bogolyubov}}
\ndef{\Bonsall}{\myauthor{F.\,F.\,Bonsall}}
\ndef{\Bony}{\myauthor{J.\,F.\,Bony}}
\ndef{\BoosBavnbek}{\myauthor{B.\,Boo$\beta$-Bavnbek}}
\ndef{\Bott}{\myauthor{R.\,Bott}}
\ndef{\Branges}{\myauthor{L.\,de Branges}}
\ndef{\Bratteli}{\myauthor{O.\,Bratteli}}
\ndef{\Bredon}{\myauthor{G.\,E.\,Bredon}}
\ndef{\Breuer}{\myauthor{M.\,Breuer}}
\ndef{\Brown}{\myauthor{L.\,G.\,Brown}}
\ndef{\Bruneau}{\myauthor{V.\,Bruneau}}
\ndef{\Buslaev}{\myauthor{V.\,S.\,Buslaev}}
\ndef{\Carey}{\myauthor{A.\,L.\,Carey}}
\ndef{\CareyRW}{\myauthor{R.\,W.\,Carey}} %Richard
\ndef{\Cartan}{\myauthor{H.\,Cartan}}
\ndef{\Chilin}{\myauthor{V.\,I.\,Chilin}}
\ndef{\Coburn}{\myauthor{L.\,A.\,Coburn}}
\ndef{\Connes}{\myauthor{A.\,Connes}}
\ndef{\Cornfeld}{\myauthor{I.\,P.\,Cornfeld}}
\ndef{\Daletskii}{\myauthor{Yu.\,L.\,Daletski\u\i}}   %Daletskii
\ndef{\Dixmier}{\myauthor{J.\,Dixmier}}
\ndef{\DoddsPG}{\myauthor{P.\,G.\,Dodds}}
\ndef{\DoddsTK}{\myauthor{T.\,K.\,Dodds}}
\ndef{\Douglas}{\myauthor{R.\,G.\,Douglas}}
\ndef{\Dubrovin}{\myauthor{B.\,A.\,Dubrovin}}
\ndef{\Dugundji}{\myauthor{J.\,Dugundji}}
\ndef{\Duncan}{\myauthor{J.\,Duncan}}
\ndef{\Dunford}{\myauthor{N.\,Dunford}}
\ndef{\Dykema}{\myauthor{K.\,J.\,Dykema}}
\ndef{\Edwards}{\myauthor{R.\,E.\,Edwards}}
\ndef{\Eilenberg}{\myauthor{S.\,Eilenberg}}
\ndef{\Entina}{\myauthor{S.\,B.\,\`Entina}}
\ndef{\Fack}{\myauthor{T.\,Fack}} %Thierry
\ndef{\Faddeev}{\myauthor{L.\,D.\,Faddeev}}
\ndef{\Farber}{\myauthor{M.\,Farber}}
\ndef{\Farforovskaya}{\myauthor{Yu.\,B.\,Farforovskaya}}
\ndef{\Federer}{\myauthor{H.\,Federer}}
\ndef{\Fedosov}{\myauthor{B.\,V.\,Fedosov}}
\ndef{\Figiel}{\myauthor{T.\,Figiel}} %Tadeush?
\ndef{\Figueroa}{\myauthor{H.\,Figueroa}}
\ndef{\Fillmore}{\myauthor{P.\,A.\,Fillmore}}
\ndef{\Fomenko}{\myauthor{A.\,T.\,Fomenko}} % Anatolii Trofimovich
\ndef{\Fomin}{\myauthor{S.\,V.\,Fomin}}
\ndef{\Frohlich}{\myauthor{J.\,Fr\"ohlich}}
\ndef{\Fuglede}{\myauthor{B.\,Fuglede}}
\ndef{\Furutani}{\myauthor{K.\,Furutani}}
\ndef{\Gelfand}{\myauthor{I.\,M.\,Gelfand}}
\ndef{\Gesztesy}{\myauthor{F.\,Gesztesy}}     %Fritz
\ndef{\Getzler}{\myauthor{E.\,Getzler}} % Ezra
\ndef{\Gilkey}{\myauthor{P.\,B.\,Gilkey}}
\ndef{\Gitler}{\myauthor{S.\,Gitler}}
\ndef{\Glazman}{\myauthor{I.\,M.\,Glazman}}
\ndef{\Glimm}{\myauthor{J.\,Glimm}}
\ndef{\Gohberg}{\myauthor{I.\,C.\,Gohberg}}
\ndef{\Goldshtein}{\myauthor{Ya.\,Goldshtein}}
\ndef{\Golze}{\myauthor{F.\,Golze}}
\ndef{\GraciaBondia}{\myauthor{J.\,M.\,Gracia-Bond\'{i}a}}
\ndef{\Greenleaf}{\myauthor{F.\,P.\,Greenleaf}}
\ndef{\Gromov}{\myauthor{M.\,Gromov}}
\ndef{\Gunning}{\myauthor{R.\,C.\,Gunning}}
\ndef{\Haagerup}{\myauthor{U.\,Haagerup}}
\ndef{\Haag}{\myauthor{R.\,Haag}}
\ndef{\Halmos}{\myauthor{Halmos}}
\ndef{\Hardy}{\myauthor{G.\,H.\,Hardy}}
\ndef{\Herbst}{\myauthor{I.\,W.\,Herbst}}
\ndef{\Higson}{\myauthor{N.\,Higson}}  % Nigel
\ndef{\Hoermander}{\myauthor{L.\,Hoermander}} % Lars
\ndef{\Hoffman}{\myauthor{K.\,Hoffman}} % Kenneth Hoffman
\ndef{\Ito}{\myauthor{K.\,Ito}}
\ndef{\Jaffe}{\myauthor{A.\,Jaffe}}
\ndef{\James}{\myauthor{I.\,M.\,James}}
\ndef{\Javrjan}{\myauthor{V.\,A.\,Javrjan}}
\ndef{\Jitomirskaya}{\myauthor{S.\,Jitomirskaya}}
\ndef{\Kadison}{\myauthor{R.\,V.\,Kadison}}
\ndef{\Kalton}{\myauthor{N.\,J.\,Kalton}} % Nigel
\ndef{\Kato}{\myauthor{T.\,Kato}} % Tosio
\ndef{\Kobayashi}{\myauthor{S.\,Kobayashi}}
\ndef{\Koplienko}{\myauthor{L.\,S.\,Koplienko}}
\ndef{\Korotyaev}{\myauthor{E.\,Korotyaev}}
\ndef{\Kosaki}{\myauthor{H.\,Kosaki}}
\ndef{\Kostrykin}{\myauthor{V.\,Kostrykin}}
\ndef{\Kotani}{\myauthor{S.\,Kotani}}
\ndef{\Krein}{\myauthor{Kre\u\i n}}
\ndef{\KreinMG}{\myauthor{M.\,G.\,Kre\u\i n}}
\ndef{\KreinSG}{\myauthor{S.\,G.\,Kre\u\i n}}
\ndef{\Kuroda}{\myauthor{S.\,T.\,Kuroda}}
\ndef{\Leichtnam}{\myauthor{E.\,Leichtnam}}
\ndef{\Lesch}{\myauthor{M.\,Lesch}}
\ndef{\Lesniewski}{\myauthor{A.\,Lesniewski}}
\ndef{\Levitan}{\myauthor{B.\,M.\,Levitan}}
\ndef{\Lidskii}{\myauthor{V.\,B.\,Lidskii}}
\ndef{\Lifshits}{\myauthor{I.\,M.\,Lifshits}}
\ndef{\Lindenstrauss}{\myauthor{J.\,Lindenstrauss}}
\ndef{\Loday}{\myauthor{J.-L.\,Loday}}
\ndef{\Lord}{\myauthor{S.\,Lord}}      %Steven
\ndef{\Lorentz}{\myauthor{G.\,Lorentz}}
\ndef{\Magnus}{\myauthor{W.\,Magnus}}
\ndef{\Makarov}{\myauthor{K.\,A.\,Makarov}}
\ndef{\MakarovN}{\myauthor{N.\,Makarov}}
\ndef{\Mathai}{\myauthor{V.\,Mathai}}         %Varghese?
\ndef{\McKean}{\myauthor{H.\,P.\,McKean}}
\ndef{\Mishchenko}{\myauthor{A.\,S.\,Mishchenko}}
\ndef{\Molchanov}{\myauthor{S.\,A.\,Molchanov}}
\ndef{\Moore}{\myauthor{C.\,C.\,Moore}}
\ndef{\Moscovici}{\myauthor{H.\,Moscovici}}  %Henri?
\ndef{\Motovilov}{\myauthor{A.\,K.\,Motovilov}}
\ndef{\Moyer}{\myauthor{R.\,D.\,Moyer}}
\ndef{\Naboko}{\myauthor{S.\,N.\,Naboko}}
\ndef{\Narasimhan}{\myauthor{R.\,Narasimhan}}
\ndef{\Nomizu}{\myauthor{K.\,Nomizu}}
\ndef{\Novikov}{\myauthor{S.\,P.\,Novikov}}
\ndef{\Osterwalder}{\myauthor{K.\,Osterwalder}}
\ndef{\Patodi}{\myauthor{V.\,Patodi}}
\ndef{\Pagter}{\myauthor{B.\,de~Pagter}}  %Ben
\ndef{\Pastur}{\myauthor{L.\,A.\,Pastur}}  %Ben
\ndef{\Pavlov}{\myauthor{B.\,S.\,Pavlov}}
\ndef{\Pedersen}{\myauthor{G.\,K.\,Pedersen}}
\ndef{\Peller}{\myauthor{V.\,V.\,Peller}}
\ndef{\Perera}{\myauthor{V.\,S.\,Perera}}
\ndef{\Petunin}{\myauthor{Ju.\,I.\,Petunin}}
\ndef{\Phillips}{\myauthor{J.\,Phillips}}  %John
\ndef{\Piazza}{\myauthor{P.\,Piazza}}   %Paolo
\ndef{\Pincus}{\myauthor{J.\,D.\,Pincus}}   %Joel
\ndef{\Poincare}{Poincar\'e}
\ndef{\Postnikov}{\myauthor{M.\,M.\,Postnikov}} % Mikhail
\ndef{\Prinzis}{\myauthor{R.\,Prinzis}}
\ndef{\Privalov}{\myauthor{I.\,I.\,Privalov}}
\ndef{\Pushnitski}{\myauthor{A.\,B.\,Pushnitski}} % Alexander
\ndef{\Raeburn}{\myauthor{I.\,Raeburn}}
\ndef{\Raikov}{\myauthor{G.\,Raikov}}
\ndef{\Reed}{\myauthor{M.\,Reed}}
\ndef{\Rennie}{\myauthor{A.\,Rennie}}
\ndef{\Rickart}{\myauthor{C.\,E.\,Rickart}}
\ndef{\Riesz}{\myauthor{F.\,Riesz}}
\ndef{\Ringrose}{\myauthor{J.\,Ringrose}}
\ndef{\Rio}{\myauthor{R.\,del Rio}}
\ndef{\Robinson}{\myauthor{D.\,Robinson}}
\ndef{\Rossi}{\myauthor{H.\,Rossi}}
\ndef{\Rudin}{\myauthor{W.\,Rudin}}
\ndef{\Ruelle}{\myauthor{D.\,Ruelle}}
\ndef{\Ruzhansky}{\myauthor{M.\,Ruzhansky}}
\ndef{\Sakai}{\myauthor{Sh.\,Sakai}}
\ndef{\Sargsjan}{\myauthor{I.\,S.\,Sargsjan}}
\ndef{\Sato}{\myauthor{H.\,Sato}}
\ndef{\Schaeffer}{\myauthor{D.\,G.\,Schaeffer}}
\ndef{\Schluchtermann}{\myauthor{G.\,Schluchtermann}}
\ndef{\Schochet}{\myauthor{C.\,Schochet}}
\ndef{\SchroedingerE}{\myauthor{E.\,Schr\"odinger}}
\ndef{\Schroedinger}{\myauthor{Schr\"odinger}}
\ndef{\Schrohe}{\myauthor{E.\,Schrohe}}
\ndef{\Schwartz}{\myauthor{J.\,T.\,Schwartz}}
\ndef{\Sedaev}{\myauthor{A.\,A.\,Sedaev}}
\ndef{\Seiler}{\myauthor{R.\,Seiler}}
\ndef{\Semenov}{\myauthor{E.\,M.\,Semenov}}
\ndef{\Shabat}{\myauthor{B.\,V.\,Shabat}}
\ndef{\Shafarevich}{\myauthor{I.\,R.\,Shafarevich}}
\ndef{\Sharpley}{\myauthor{R.\,Sharpley}}
\ndef{\Shilov}{\myauthor{G.\,E.\,Shilov}}
\ndef{\Shirkov}{\myauthor{D.\,V.\,Shirkov}}
\ndef{\Shubin}{\myauthor{M.\,A.\,Shubin}}
\ndef{\Silverman}{\myauthor{H.\,Silverman}}
\ndef{\Simon}{\myauthor{B.\,Simon}}
\ndef{\Sinai}{\myauthor{Ya.\,G.\,Sinai}}
\ndef{\Singer}{\myauthor{I.\,M.\,Singer}}
\ndef{\Solomyak}{\myauthor{M.\,Z.\,Solomyak}}
\ndef{\Soloviev}{\myauthor{Yu.\,P.\,Soloviev}}
\ndef{\Spivak}{\myauthor{M.\,Spivak}}
\ndef{\Stein}{\myauthor{E.\,M.\,Stein}}
\ndef{\Stenkin}{\myauthor{V.\,V.\,Sten'kin}}
\ndef{\Stratila}{\myauthor{S.\,Stratila}}
\ndef{\Sucheston}{\myauthor{L.\,Sucheston}}
\ndef{\Sukochev}{\myauthor{F.\,A.\,Sukochev}}
\ndef{\Switzer}{\myauthor{R.\,M.\,Switzer}}
\ndef{\SzNagy}{\myauthor{B.\,Sz.-Nagy}}
\ndef{\Takesaki}{\myauthor{M.\,Takesaki}}
\ndef{\Taylor}{\myauthor{M.\,E.\,Taylor}}
\ndef{\Treves}{\myauthor{F.\,Treves}}
\ndef{\Troitsky}{\myauthor{E.\,V.\,Troitsky}}
\ndef{\Tzafriri}{\myauthor{L.\,Tzafriri}}
\ndef{\Varilly}{\myauthor{J.\,C.\,V\'{a}rilly}}
\ndef{\Vergne}{\myauthor{M.\,Vergne}}
\ndef{\Vladimirov}{\myauthor{V.\,S.\,Vladimirov}}
\ndef{\Voiculescu}{\myauthor{D.\,Voiculescu}}
\ndef{\Weiss}{\myauthor{G.\,Weiss}}
\ndef{\Wells}{\myauthor{R.\,O.\,Wells}}
\ndef{\Williams}{\myauthor{J.\,P.\,Williams}}
\ndef{\Winkler}{\myauthor{S.\,Winkler}}
\ndef{\Witten}{\myauthor{E.\,Witten}}
\ndef{\Wodzicki}{\myauthor{M.\,Wodzicki}}
\ndef{\Wojciechowski}{\myauthor{K.\,P.\,Wojciechowski}}
\ndef{\Yafaev}{\myauthor{D.\,R.\,Yafaev}}
\ndef{\Yosida}{\myauthor{K.\,Yosida}}
\ndef{\Zsido}{\myauthor{L.\,Zsido}}

\ndef{\hlambda}{{\mathfrak h_\lambda}}
\ndef{\hlambdao}{{\mathfrak h_\lambda^{(0)}}}
\ndef{\hlambdar}{{\mathfrak h_\lambda^{(r)}}}

\begin{document}
\title{Singular spectral shift is additive}
\author{\Azamov}
\address{School of Computer Science, Engineering and Mathematics
   \\ Flinders University
   \\ Bedford Park, 5042, SA Australia.}
\email{azam0001@csem.flinders.edu.au}
% \keywords{Spectral shift function, scattering matrix, Birman-\Krein\ formula}

\subjclass[2000]{ %Mathematics Subject Classification (2000).
    Primary 47A55; % Perturbation theory
    % Secondary 47A11 % Local spectral properties
}
\begin{abstract} In this note it is proved that the singular part of the spectral
shift function is additive. That is, if $H_0,H_1$ and $H_2$ are self-adjoint (not necessarily bounded) operators
with trace-class differences, then
$$
  \xis_{H_2,H_0} = \xis_{H_2,H_1} + \xis_{H_1,H_0}.
$$
Here, for any $\phi \in C_c(\mbR)$
$$
  \xis_{H_1,H_0}(\phi) := \int_0^1 \Tr(V \phi(H_r^{(s)}))\,dr,
$$
where $V = H_1-H_0,$ $H_r = H_0+rV$ and $H_r^{(s)}$ is the singular part of $H_r.$
\end{abstract}
\maketitle
%\begin{center} {\small \sf\today}  \\ \bigskip Preliminary DRAFT \\ \bigskip v1.0 \end{center}

\section{Introduction}

Let $H_0$ be a self-adjoint operator and $V$ be a trace class self-adjoint
operator. The Lifshits-Krein spectral shift function (\cite{Li52UMN,Kr53MS}, see also \cite[Chapter 8]{Ya} and \cite{SimTrId2}) is the unique $L_1$-function
$\xi_{H_0+V,H_0}$ such that for any $\phi \in C_c^\infty$
the equality
$$
  \Tr(\phi(H_0+V) - \phi(H_0)) = \int \xi_{H_0+V,H_0}(\lambda)\phi'(\lambda)\,d\lambda.
$$
holds. Krein also showed in \cite{Kr53MS} that for any self-adjoint operators $H_0,H_1$ and $H_2$
with trace-class differences the equality
$$
  \xi_{H_2,H_0} = \xi_{H_2,H_1} + \xi_{H_1,H_0}
$$
holds.

In \cite{BS75SM}, Birman and Solomyak proved the following spectral averaging formula for the spectral shift function:
$$
  \xi_{H_0+V,H_0}(\phi) := \int_0^1 \Tr(V \phi(H_r))\,dr, \ \phi \in C_c(\mbR)
$$
(note that if $\phi$ is a function then $\xi$ in $\xi(\phi)$ denotes a measure, and if $\lambda$ is a number then
$\xi$ in $\xi(\lambda)$ denotes a function --- density of the absolutely continuous measure $\xi$).

In \cite{Az3v4} (see also \cite{Az,Az2}) I introduced the so-called absolutely continuous and singular spectral
shift functions $\xia$ and $\xis$ by formulas
$$
  \xia_{H_0+V,H_0}(\phi) := \int_0^1 \Tr(V \phi(H_r^{(a)}))\,dr, \ \phi \in C_c(\mbR)
$$
and
$$
  \xis_{H_0+V,H_0}(\phi) := \int_0^1 \Tr(V \phi(H_r^{(s)}))\,dr, \ \phi \in C_c(\mbR),
$$
where $H_r = H_0+rV,$ $H_r^{(a)}$ is the absolutely continuous part of $H_r$ and $H_r^{(s)}$ is the singular part of $H_r.$

The distributions $\xis$ and $\xia$ are absolutely continuous finite measures \cite{Az3v4}.

In \cite{Az3v4} it is proved that for all
operators $V_1$ from a linear manifold $\clA_0 \subset \clL_1,$
which is dense in $\clL_1,$ the equality
\begin{equation} \label{F: xia is additive}
  \xia_{H_0+V,H_0}(\phi) = \xia_{H_0+V,H_0+V_1}(\phi) +
  \xia_{H_0+V_1,H_0}(\phi).
\end{equation}
holds for all $\phi \in C_c^\infty.$ This equality implies similar
equality for $\xis.$

In this note I give a proof of the equality (\ref{F: xia is additive}) for all trace-class self-adjoint operators $V$ and $V_1.$
This implies that for any self-adjoint operator $H_0$ and any trace-class self-adjoint operators $V_1$ and $V_2$
the equality
\begin{equation} \label{F: xis is additive}
  \xis_{H_0+V_2,H_0}(\phi) = \xis_{H_0+V_2,H_0+V_1}(\phi) + \xis_{H_0+V_1,H_0}(\phi)
\end{equation}
holds.

The additivity property (\ref{F: xis is additive}) of the singular spectral shift function $\xis$ combined
with the fact that the density $\xis(\lambda)$ of the measure $\xis$ is a.e. integer-valued \cite{Az3v4},
suggests that the singular spectral shift function should be interpreted as generalization of spectral flow
of eigenvalues (see e.g. \cite{APS76,Ge93Top,Ph96CMB,Ph97FIC,CP98CJM,CP2,ACDS,ACS,Azbook}) to the case of spectral flow inside the essential spectrum.

\section{Results}

\begin{thm} \label{T: xia is weakly cont-s w.r.t. V} Let $H_0$ be a self-adjoint operator on $\hilb,$ let $V$ be a trace-class self-adjoint operator on $\hilb.$
If $V_1, V_2, \ldots$ is a sequence of self-adjoint trace-class operators converging to $V$ in the trace-class norm, then
for any $\phi \in C_c$ the equality
$$
  \lim_{n \to \infty} \xi^{(a)}_{H_0+V_n,H_0}(\phi) = \xi^{(a)}_{H_0+V,H_0}(\phi).
$$
holds. Shortly, the absolutely continuous part of the spectral shift function  $\xi^{(a)}_{H_0+V,H_0}$ is weakly-continuous with respect to $V \in \clL_1(\hilb).$
\end{thm}
\begin{proof}
%%So, let, as in Theorem \ref{T: xia is weakly cont-s w.r.t. V}, $H_0$ be a self-adjoint operator on $\hilb,$ let $V$ be a trace-class self-adjoint operator on $\hilb,$
%%and let $V_1, V_2, \ldots$ be a sequence of self-adjoint trace-class operators converging to $V$ in the trace-class norm.
%We have to prove that for any $\phi \in C_c^\infty$ the equality
%$$
%  \lim_{n \to \infty} \xi^{(a)}_{H_0+V_n,H_0}(\phi) = \xi^{(a)}_{H_0+V,H_0}(\phi).
%$$
%holds. That is,
We have to prove that for any $\phi \in C_c(\mbR)$ the difference
\begin{equation} \label{F: 33}
  \int_0^1 \brs{\Tr\brs{V \phi(H_0+rV)^{(a)}} - \Tr \brs{V_n \phi(H_0+rV_n)^{(a)}}}\,dr
\end{equation}
goes to $0$ as $n\to \infty.$ Since the integrand as a function of $r$ is bounded by
$2\norm{V}_1 \norm{\phi}_\infty$ for all large enough $n,$  it follows from the Lebesgue dominated convergence theorem
that it is enough to prove that for any fixed $r \in [0,1]$
$$
  \lim_{n \to \infty} \Tr \brs{V_n \phi(H_0+rV_n)^{(a)}} = \Tr\brs{V \phi(H_0+rV)^{(a)}}.
$$
Further, since
\begin{equation*}
  \begin{split}
       & \Tr\brs{V \phi(H_0+rV)^{(a)}} - \Tr \brs{ V_n \phi(H_0+r V_n)^{(a)}}
       \\ & \qquad = \Tr\brs{(V-V_n) \phi(H_0+rV_n)^{(a)}} + \Tr \brs{V \brs{\phi(H_0+r V)^{(a)} - \phi(H_0+ rV_n)^{(a)}} }
       %\\ & \qquad = \brs{V - V_n}  + V \brs{\phi(H_0+rV)^{(a)} - \phi(H_0+r V_n)^{(a)}},
  \end{split}
\end{equation*}
and since
$$
  \abs{\Tr\brs{(V-V_n) \phi(H_0+rV_n)^{(a)}}} \leq \norm{V-V_n}_1 \cdot \norm{\phi}_\infty \to 0 \ \text{as} \ n \to \infty,
$$
it is enough to prove that
\begin{equation} \label{F: something enough to prove}
  \lim_{n \to \infty} \Tr \brs{V \brs{\phi(H_0+r V)^{(a)} - \phi(H_0+ rV_n)^{(a)}} } = 0.
\end{equation}
It follows from \cite[Lemma 6.1.3]{Ya}, that for this it is enough to show that
\begin{equation} \label{F: 77}
  \text{s-}\lim_{n \to \infty} \phi(H_0+ rV_n)^{(a)} = \phi(H_0+rV)^{(a)},
\end{equation}
where the limit is taken in the strong operator topology.
We can assume that $r = 1.$
Let $H = H_0+V,$ $H_n=H_0+V_n.$
For self-adjoint operators $H_0$ and $H_1,$
let $W_\pm(H_1,H_0)$ be wave operators of the pair $H_0$ and $H_1$ (if they exist) and let $P^{(a)}(H_0)$
be the orthogonal projection onto the absolutely continuous part of $H_0.$
Since
$$
  W_+(H_n,H) \phi(H^{(a)}) W_+^*(H_n,H) =  \phi(H_n^{(a)}),
$$
it follows that
\begin{equation} \label{F: 55}
  \begin{split}
    \phi(H^{(a)}) - \phi(H_n^{(a)}) & = \phi(H^{(a)}) - W_+(H_n,H)\phi(H^{(a)})W_+^*(H_n,H)
    \\ & = \brs{\phi(H^{(a)}) - W_+(H_n,H)\phi(H^{(a)})}
    \\ & \qquad \qquad + \brs{W_+(H_n,H)\phi(H^{(a)}) - W_+(H_n,H)\phi(H^{(a)})W_+^*(H_n,H)}
    \\ & = \brs{P^{(a)}(H) - W_+(H_n,H)} \phi(H^{(a)})
    \\ & \qquad\qquad + W_+(H_n,H)\phi(H^{(a)})\brs{P^{(a)}(H) - W_+^*(H_n,H)}.
  \end{split}
\end{equation}
\cite[Theorem 6.3.6]{Ya} implies that
$$
  \text{s-}\lim_{n \to \infty} W_+(H_n,H) = P^{(a)}(H)
$$
and
$$
  \text{s-}\lim_{n \to \infty} W_+^*(H_n,H) = P^{(a)}(H).
$$
It follows from this and (\ref{F: 55}) that (\ref{F: 77}) holds.

The proof is complete.
\end{proof}

\begin{thm} \label{T: xia is additive} The absolutely continuous part of the spectral shift function is additive.
That is, if $H_0$ is a self-adjoint operator on $\hilb,$ and if $V_1,V_2$ are trace-class self-adjoint operators on $\hilb,$
then for any $\phi\in C_c(\mbR)$ the equality
\begin{equation} \label{F: 2}
  \xia_{H_0+V_2,H_0}(\phi) = \xia_{H_0+V_2,H_0+ V_1}(\phi) + \xia_{H_0+V_1,H_0}(\phi)
\end{equation}
holds.
\end{thm}
\begin{proof} Let $H_0$ be a self-adjoint operator on $\hilb,$ and let $V$ and $V_1$ be two trace-class self-adjoint operators on $\hilb.$
We need to show that for any $\phi \in C_c^\infty$
$$
  \xia_{H_0+V,H_0}(\phi) = \xia_{H_0+V,H_0+ V_1}(\phi) + \xia_{H_0+V_1,H_0}(\phi)
$$
By \cite[Lemma 5.2]{Az3v4}, for a given trace-class operator $V$ one can choose a frame operator $F$
(see \cite{Az3v4} for the definition of the frame operator) such that $V \in \clA(F) \subset \clL_1(\hilb),$
where $\clA(F)$ is a dense linear subset of $\clL_1(\hilb)$ (see \cite[\S 5]{Az3v4} for the definition of the class $\clA(F)$).

By \cite[Theorem 9.12]{Az3v4}, there exists a dense linear subset $\clA_0$ (which depends on $H_0$) of $\clA(F),$ such that for any $\tilde V \in \clA_0$
and any function $\phi \in C_c(\mbR)$ the equality
\begin{equation} \label{F: 1}
  \xia_{H_0+V,H_0}(\phi) = \xia_{H_0+V,H_0+ \tilde V}(\phi) + \xia_{H_0+\tilde V,H_0}(\phi)
\end{equation}
holds. Since $\clA_0$ is dense in $\clL_1(\hilb)$ too, it follows that there exists a sequence $V_2, V_3, \ldots \in \clA_0,$
such that $V_n \to V_1$ in the trace class norm as $n \to \infty,$ and for any $n =2,3,\ldots$
the equality
\begin{equation} \label{F: 5}
  \xia_{H_0+V,H_0}(\phi) = \xia_{H_0+V,H_0+ V_n}(\phi) + \xia_{H_0+V_n,H_0}(\phi)
\end{equation}
holds.
By Theorem \ref{T: xia is weakly cont-s w.r.t. V},
\begin{equation} \label{F: 3}
  \lim_{n \to \infty} \xia_{H_0+V_n,H_0}(\phi) = \xia_{H_0+V_1,H_0}(\phi).
\end{equation}
It directly follows from the definition of $\xia$ that
\begin{equation} \label{F: xia=-xia}
  \xia_{H_1,H_0} = - \xia_{H_0,H_1}
\end{equation}
for any two self-adjoint operators $H_0,H_1$ with trace-class difference.
% Indeed, for any $\phi \in C_c^\infty,$
% \begin{equation*}
%  \begin{split}
%   \xia_{H_1,H_0}(\phi) & = \int_0^1 \Tr\brs{\frac{dH_r}{dr} \phi(H_r^{(a)})}\,dr
%     = \int_1^0 \Tr\brs{-\frac{dG_t}{dt} \phi(G_t^{(a)})}\,(-dt)
%     \\ & = - \int_0^1 \Tr\brs{\frac{dG_t}{dt} \phi(G_t^{(a)})}\,dt
%     = - \xia_{G_1,G_0}(\phi) = - \xia_{H_0,H_1}(\phi).
%  \end{split}
% \end{equation*}
% where $G_t = H_{1-t}$ and $r = 1-t.$
It follows from (\ref{F: xia=-xia}) and Theorem \ref{T: xia is weakly cont-s w.r.t. V} that
$$
  \lim_{n \to \infty} \xia_{H_0+V,H_0+V_n}(\phi) = \xia_{H_0+V,H_0+V_1}(\phi).
$$
Combining this equality with (\ref{F: 5}) and (\ref{F: 3}) completes the proof.
\end{proof}

\begin{cor} The singular part of the spectral shift function is additive.
That is, if $H_0$ is a self-adjoint operator on $\hilb,$ and if $V_1,V_2$ are trace-class self-adjoint operators on $\hilb,$
then for any $\phi\in C_c(\mbR)$ the equality
\begin{equation*}
  \xis_{H_0+V_2,H_0}(\phi) = \xis_{H_0+V_2,H_0+ V_1}(\phi) + \xis_{H_0+V_1,H_0}(\phi)
\end{equation*}
holds.
\end{cor}
\begin{proof} This follows from Theorem \ref{T: xia is additive} and additivity of the Lifshits-Krein spectral shift function.
\end{proof}

\rndef{\emph}[1]{{\it #1}}

\mathsurround 0pt
\ndef{\AndSoOn}{$\dots$}

%\end{document}
%  End of MyListOfRef.tex

\end{document}